\batchmode
\documentclass{amsart}
\usepackage{xspace}
\usepackage{amssymb}
\usepackage{latexsym}
\usepackage{amsmath}
\usepackage{amsfonts}
\usepackage{amsthm}
\usepackage{graphicx}
\usepackage{psfrag}
\usepackage{epsfig}

\numberwithin{equation}{section}

\newtheorem{thm}{Theorem}[section]
\newtheorem{lem}[thm]{Lemma}
\newtheorem{prop}[thm]{Proposition}
\newtheorem{cor}[thm]{Corollary}
\newtheorem{conj}[thm]{Conjecture}

\newcommand{\T}{\operatorname{T}}
\newcommand{\h}{\operatorname{ht}}

\newcommand{\R}{\mathbb R}
\newcommand{\N}{\mathbb N}
\newcommand{\SL}{\operatorname{SL}}
\newcommand{\Z}{ \mathbb{Z}}

\begin{document}
\title[Invariant measures on the space of lattices]{Positive entropy invariant measures on the space of lattices with escape of mass}
\author{Shirali Kadyrov}
\thanks{The author acknowledges support by the SNF
(200021-127145).}

\begin{abstract}
On the space of unimodular lattices, we construct a sequence of invariant probability measures under a singular diagonal element with high entropy and show that the limit measure is 0.
\end{abstract}
\maketitle
\section{Introduction}
Consider the homogeneous space $X_3=\SL_{3}(\Z)\backslash\SL_{3}(\R)$ with the transformation $\T_3$ acting as a right multiplication by $diag(e^{1/2},e^{1/2},e^{-1})$. In a joint work with M.~Einsiedler in \cite{EinKad} we prove the following.
\begin{thm}
\label{thm:previous}
For any sequence of $\T_3$-invariant probability measures $\mu_i$ on $X_3$ and $c \in [2,3]$ with $h_{\mu_i}(\T_3)\ge c$ one has that any weak$^*$ limit $\mu$ of $(\mu_i)$ has $\mu(X_3)\ge c-2$.
\end{thm}
This shows that a lower bound on the entropy of a sequence of measures controls escape of mass in any weak$^*$ limit.
We say that $\mu$ is a weak$^*$ limit of the sequence $(\mu_i)_{i\ge 1}$ if for some subsequence $i_k$ and for all
$f \in C_c(X)$ we have $$\lim_{k\to \infty}\int_X f d\mu_{i_k} \to \int_X f d \mu.$$
If $c < 2$ then the theorem does not tell us whether one should expect some positive mass left. In this paper we show that actually it is possible that if $c<2$ then the limit measure could be zero, and also show this in higher dimension.

For $d \ge 1$ we let $G=\SL_{d+1}(\R)$ and $\Gamma=\SL_{d+1}(\Z)$. We consider the homogeneous space $X=\Gamma\backslash G$ and a transformation $\T$ defined by
$$\T(x)=x a$$
where $a=diag(e^{1/d},e^{1/d},...,e^{1/d},e^{-1}) \in G.$ 
\begin{thm}
\label{thm:0 measure}
There exists a sequence of $\T$-invariant probability measures $(\mu_i)_{i \ge 1}$ on $X$ whose entropies satisfy $\lim_{i\to \infty} h_{\mu_i}(\T)=d$ but the weak$^*$ limit $\mu$ is the zero measure.
\end{thm}
We note here that the maximum measure theoretic entropy, the entropy of $\T$ with respect to Haar measure on $X$, is $d+1$. This follows for example from \cite[Prop. 9.2 and 9.6]{MarTom}. 
An immediate consequence  of Theorem~\ref{thm:0 measure} is the following corollary.
\begin{cor}
\label{cor:c measure}
For any $c \in [0,1]$ there exists a sequence of $\T$-invariant probability measures $(\nu_i)_{i\ge 1}$ on $X$ whose entropies satisfy $\lim_{\to \infty}h_{\mu_i}(\T)=d+c$ such that any weak$^*$ limit has mass $c$.
\end{cor}
Theorem~\ref{thm:previous} and Corollary~\ref{cor:c measure} suggest the following.
\begin{conj}
Let $\T$ and $X$ be as above with $d \ge 3$ and let $c \in [d,d+1]$. Then for $\T$-invariant probability measures $\mu_i$ on $X$ with $h_{\mu_i}(\T)\ge c$ one has that any weak$^*$ limit $\mu$ of $(\mu_i)_{i \ge 1}$ has $\mu(X)\ge c-d$.
\end{conj}
For more general conjecture of the similar spirit we refer to \cite{Che}. There, it is stated in terms of the Hausdorff dimension of the set of points that lie on divergent trajectories for the non-quasi-unipotent flow.

Let $M>0$ be given. For a lattice $x \in X$, define the height $\h(x)$ to be the inverse of the length of the shortest nonzero vector in $x$. Also, define the sets 
$$X_{<M}=\{x \in X : \h(x) <M\} \text{ and } X_{\ge M}=\{x \in X : \h(x) \ge M\}.$$
We note that by Mahler's compactness criterion $X_{< M}$ is pre-compact.
Theorem~\ref{thm:0 measure} follows from the following.
\begin{thm}
\label{thm:measure}
For any $\epsilon>0$ and $M \ge 1$ there exists a $\T$-invariant measure $\mu$ with $h_{\mu}(\T)>d-\epsilon$ such that $\mu(X_{\geq M})>1-\epsilon$.
\end{thm}
We will construct infinitely many points in $X_{<M}$ whose forward trajectories mostly stay above height $M$. Taking union of the sets of forward trajectories of these points, we will construct a $\T$-invariant set $S_N$ with topological entropy greater than $d-\epsilon$ (cf. Theorem~\ref{thm:main}). To construct the $\T$-invariant probability measures we want, we will make use of the Variational Principle. In the next section, we introduce preliminary definitions and deduce  Theorem~\ref{thm:0 measure} and its corollary assuming Theorem~\ref{thm:measure}. In \S~\ref{sec:thm:measure} we prove Theorem~\ref{thm:measure} assuming Theorem~\ref{thm:main}. In the last two sections we prove Theroem~\ref{thm:main}.

{\bf Acknowledgments:} This work is part of the author's doctoral dissertation at The Ohio State University. The author would like to thank his adviser M. Einsiedler for encouragement and useful conversations.
\section{Preliminaries}
\subsection{Topological Entropy and Variational Principle}
In this section we will briefly introduce topological entropy and its relation to measure theoretic entropy which is called the Variational Principle. For details and proofs we refer to Chapter 7 and Chapter 8 of \cite{WB}.

There are various definitions of topological entropy. Here, we will give the definition of topological entropy in terms of separated sets. Let $(Y,d_0)$ be a compact metric space and let $\T:Y \to Y$ be a continuous map. Define a new metric $d_n$ on $Y$ by
$$d_n(x,y)=\max_{0\leq i \leq n-1}d_0(\T^i(x),\T^i(y)).$$
For a given $\epsilon>0$ and a natural number $n$, we say that the couple $x,y$ is $(n,\epsilon)$\textit{-separated} if $d_n(x,y) \ge \epsilon$ and we say that the set $E$ is $(n,\epsilon)$\textit{-separated} if any distinct $x,y \in E$ is $(n,\epsilon)$-separated.

Now define $s_n(\epsilon,Y)$ to be the cardinality of the largest possible $(n,\epsilon)$-separated set and let
$$s(\epsilon,Y)=\limsup _{n \to \infty}\frac{1}{n}\log s_n(\epsilon,Y).$$
Finally, we define the \textit{topological entropy} of $\T$ with respect to $Y$ by
$$h(\T)=\lim_{\epsilon \to 0}s(\epsilon,Y).$$
Here is the relation between the topological entropy and measure theoretic entropy:
\begin{thm}[Variational Principle]
\label{thm:var}
Topological entropy $h_{\T}(Y)$ of a $\T$-invariant compact metric space $Y$ is the supremum of measure theoretic entropies $h_{\mu}(Y)$ where supremum is taken over all $\T$-invariant probability measures on the set $Y$.
\end{thm}
\subsection{Riemannian metric on $X$}
\label{inj}
Let $G=\SL_{d+1}(\R)$ and $\Gamma=\SL_{d+1}(\Z)$. We fix a left-invariant Riemannian metric $d_G$ on $G$ and for any $x_1=\Gamma g_1,x_2=\Gamma g_2 \in X$ we define
$$d_X(x_1,x_2)=\inf_{\gamma \in \Gamma} d_{G}(g_1,\gamma g_2)$$
which gives a metric $d_X$ on $X=\Gamma \backslash G.$  For more information about the Riemannian metric, we refer \cite[Chp. 2]{RM}.
\subsubsection{Injectivity radius}
Let $B_r^H(x):=\{h \in H\,|\,d(h,x)<r\}$ where $d$ is a metric defined in $H$ and $B_r^H$ is understood to be $B_r^H(1)$.
\begin{lem}
For any $x\in X$ there is an injectivity radius $r>0$ such that the map $g \mapsto xg$ from $B_r^G \to B_r^X(x)$ is an isometry.
\end{lem}
Note that since $X_{<M}$ is pre-compact we can choose a uniform $r>0$ which is an injectivity radius for every point in  $X_{<M}$. In this case, $r$ is called \textit{an injectivity radius of $X_{<M}$}.
\subsection{Relations between the metrics}
\label{sec:Operatornorms}
We endow $\R^d$, $\R^{d+1}$, and $\R^{(d+1)^2}$ with the maximum norm $\|\cdot\|$. Rescaling the Riemannian metric if necessary we will assume that there exists $\eta_0\in (0,1)$ and $c_0>1$ such that 
\begin{equation}
\label{eqn:metric}
d_G(1,g)<\|1-g\|<c_0d_G(1,g) 
\end{equation}
for any $g \in B_{\eta_0}^G.$
\subsection{Some deductions}
Now we will deduce Corollary~\ref{cor:c measure} from Theorem~\ref{thm:0 measure} and prove Theorem~\ref{thm:0 measure} assuming Theorem~\ref{thm:measure}.
\begin{proof}[Proof of Corollary~\ref{cor:c measure}]
Let $\{\mu_i\}$ be as in Theorem~\ref{thm:0 measure} and let $\lambda$ be the Haar measure on $X$. We know that $h_\lambda(\T)=d+1$ which is the maximum entropy. This follows for example  from   \cite[Prop. 9.2 and 9.6]{MarTom}. Define $\nu_i=c \lambda+(1-c)\mu_i.$ Then we have $h_{\nu_i}(\T)=c h_{\lambda}(\T)+(1-c)h_{\mu_i}(\T)$ so that
$\lim_{i\to \infty } h_{\nu_i}(\T)= d+c.$
On the other hand, $\lim_{i \to \infty} \nu_i=c \lambda$. Hence, limiting measure has $c$ mass left.
\end{proof}
\begin{proof}[Proof of Theorem~\ref{thm:0 measure}]
Now, let us assume Theorem~\ref{thm:measure}. For any natural number $i$, we let $\mu_i$ to be the $\T$-invariant measure with $h_{\mu_i}>d-\frac{1}{i}$ such that $\mu_i(X_{\geq i})>1-\frac{1}{i}$ then any weak$^*$ limit has mass 0.
\end{proof}
\section{The proof of Theorem \ref{thm:measure}}
\label{sec:thm:measure}
Before we start the construction, we would like to deduce Theorem~\ref{thm:measure} from Theorem~\ref{thm:main} below.

Let $\delta>0$ be an injectivity radius for $X_{<17M}$ with $\delta<\min\{\frac{1}{8M},\eta_0\}.$ Here is an easy lemma which will be used repeatedly in the last section.
\begin{lem}
\label{lem:join}
There exists $N'> 0$ such that for any $x,y \in X_{<17M}$ there exists $z \in X_{<17M}$ such that $d(z,y)<\delta/(c_0^3 3^{9})$ and $d(x,\T^{N'}(z))< \delta/(c_0^3 3^{9})$.
\end{lem}
\begin{proof}
Let $\lambda$ be the Haar measure on $X$. Since $X_{<17M}$ is precompact we can cover it with open balls $\mathcal{O}_1,\mathcal{O}_2,\dots, \mathcal{O}_k$ of diameter $\delta/(c_0^3 3^{9})$. They have positive measure with respect to the Haar measure. Since $\T$ is mixing with respect to the Haar measure, for any $i,j \in \{1,2,...,k\}$ there exists $N_{ij}\geq 0$ with $\lambda(\T^{-l}(\mathcal{O}_j) \cap \mathcal{O}_i) >0$ for any $l\ge N_{ij}$. Letting $N'=\max\{N_{ij}:i,j=1,2,...,k \}$ we obtain the lemma.
\end{proof}
For a given $M\ge 1$ we fix $N'$ as in Lemma~\ref{lem:join}.
 \begin{thm}
 \label{thm:main}
 Let $M \ge 1$ be given. For any large $N$ let $K=\lfloor \frac{1}{13} e^{dN} \rfloor$. Then there exist a constant $M'>1$ and a set $S_N$ in $X_{<M}$ such that  
 $$\T^l(x) \in X_{<M'} \text{ for all } x\in S_N \text{ and for all } l \ge 0. $$
 Moreover, there exists a constant $s>0$ such that for any $m \in \N$ there are subsets $S_N(m)$ of $S_N$ with the following properties:
 \begin{enumerate}
 \item cardinality of $S_N(m)$ is $K^m$
 \item $S_N(m)$ is $(m N+ (m-1)N',s)$-separated and
 \item for any $x\in S_N(m)$ we have 
 $$|\{l\in [0,m N+ (m-1)N']:\T^l(x)\in X_{\ge M/(c_0+1)}\}|\ge m N.$$
 \end{enumerate}
 \end{thm}
 Now we deduce Theorem~\ref{thm:measure} from Theorem~\ref{thm:main}.
\begin{proof}[Proof of the Theorem \ref{thm:measure}]
Let $\epsilon >0$ be given and let $N'$ be as in Lemma~\ref{lem:join}. Choose $N$ large enough so that 
$$\frac{1}{N+N'}\log \lfloor \frac{1}{13} e^{dN} \rfloor>d-\epsilon \text{ and }  \frac{N'}{N+N'}<\epsilon$$
 and let $S_N$ be the set as in Thereom~\ref{thm:main}.

To obtain a $\T$-invariant probability measure with high entropy we would like to make use of Variational Principle \ref{thm:var}. For this, we need a compact $\T$-invarinat subspace of $X$. We define
$$Y_{\leq {M'}}=\{x \in X_{\leq {M'}} \,\,|\,\, \T^l(x) \in X_{\leq M'},\text{ for } l\ge 0\}.$$
Clearly, we obtain a $\T$-invariant compact subspace containing $\T^l(S_N)$ for all $l\ge 0$.

We have $h_{\T}(Y_{\leq {M'}})>d-\epsilon$ since $Y_{\leq M'}$ contains the sets $S_N(m)$ which are $(m N+ (m-1)N',s)$-separated by Theorem \ref{thm:main}. Now, from Variational Principle~\ref{thm:var} we know that there is a $\T$-invariant measure $\mu$ on $Y_{\leq {M'}}$, hence on $X$, with $h_{\mu}(\T)>d-\epsilon$. In order to obtain the theorem, we want to have $\mu(X_{\geq M/(c_0+1)})>1-\epsilon$, but we do not get this from Variational Principle itself. Thus, we need to look into the proof of Variational Principle and see how the measures are constructed.

Let $S_N(m)$ be the subset of $Y_{\leq {M'}}$ as in Theorem \ref{thm:main}. We have that $S_N(m)$ is $(m N+ (m-1)N',s)$-separated and has cardinality $K^m$ where $K=\lfloor \frac{1}{13} e^{dN} \rfloor$. Define a probability measure
$$\sigma_m=\frac{1}{K^m}\sum_{x \in S_N(m)}\delta_x
\text{ where }\delta_x(A)=\left\{\begin{array}{c}
                       1 \text{ if } x \in A \\
                       0 \text{ if } x \not \in A
                     \end{array}
\right ..$$
Now, let a probability measure $\mu_m$ be defined by
$$\mu_m=\frac{1}{mN+(m-1)N'}\sum_{i=0}^{mN+(m-1)N'-1}\sigma_m \circ \T^{-i}$$
 where  $\sigma_m \circ \T^{-i}(A)=\sigma_m(\T^{-i}(A))$ for any measurable set $A$.
  We know that $\mathcal{M}(Y_{\leq {M'}})$, the space of Borel probability measures, is compact in the weak$^*$ topology \cite[Theorem 6.5]{WB}. We obtained a set of measures $\mu_m \in \mathcal{M}(Y_{\leq {M'}})$. If necessary going into subsequence, we have that $\{\mu_m\}$ converges to some probability measure $\mu$ in $\mathcal{M}(Y_{\leq {M'}})$. The measure $\mu$ we obtained is $\T$-invariant \cite[Theorem 6.9]{WB}. From the proof of Variational Principle \cite[Theorem 8.6]{WB}, we know that $\mu$ has
  \begin{align*}
  h_{\mu}(\T_{|Y_{\leq M'}})& \geq \lim_{m \to \infty}\frac{1}{mN+(m-1)N'} \log s_m(\epsilon,Y_{\leq M'})\\
   &\geq \lim_{m \to \infty}\frac{1}{mN+(m-1)N'}\log K^m\\
   &= \frac{1}{N+N'}\log K.
   \end{align*}
   On the other hand, by assumption we have $\frac{1}{N+N'}\log K  >d-\epsilon$ and hence we obtain $$h_{\mu}(\T)\geq h_{\mu}(\T_{|Y_{\leq {M'}}})>d-\epsilon.$$

We have $\mu_m(X_{<M/(c_0+1)})=\frac{1}{mN+(m-1)N'}\sum_{i=0}^{mN+(m-1)N'-1}\sigma_m \circ \T^{-i}(X_{<M/(c_0+1)}).$ Hence, from part $(iii)$ of Theorem~\ref{thm:main}
$$\mu_m(X_{< M/(c_0+1)}) \leq \frac{(m-1)N'}{mN+(m-1)N'} < \frac{N'}{N+N'}<\epsilon.$$
It is easy to see, approximating $X_{<M/(c_0+1)}$ by continuous functions with compact support, that 
$$\mu(X_{\ge M/(c_0+1)})>1-\epsilon. $$
So, we obtain the theorem if we apply Theorem~\ref{thm:main} for $(c_0+1)M$ instead of $M$.
 \end{proof}
\section{Initial setup and shadowing lemma}
\label{sec:initialstep}
In this section we will construct about $e^{dN}  $ lattices whose forward trajectories stay above height $M$ in the time interval $[1,N]$ for some large number $N$. Later we prove the shadowing lemma \ref{lem:shadow}, which will be used in the proof of Theorem~\ref{thm:main} in the next section.

Fix a height $M>0$. Let $N \in \mathbb N$ be a given. For $t=(t_1,t_2,...,t_d) \in [0,e^{-N/d}]^d$ consider the lattice $x_t=\Gamma g_t$ where
\begin{equation}
\label{eqn:g_t}
g_t=  \left( \begin{array}{ccccc}  M^{1/d}& 0&...&0&0 \\
 0&M^{1/d} &...&0&0\\
\vdots&\vdots&&\vdots&\vdots\\
0&0 &...&M^{1/d}&0\\
\frac{t_1}{M} &\frac{t_2}{M}&...&\frac{t_d}{M}&\frac{1}{M}  \end{array} \right).
\end{equation}
We would like to consider those lattices that stay above height $M$ in $[1,N]$ and are in $X_{<16M}$ at time $N$. We start with first considering the set
$$A_N:=\{ t \in [0, e^{-N/d}]^d : \T^{N}(x_t) \in X_{<16M}\}.$$
We claim that $A_N$ is significant in size.
\begin{lem}
For $d \ge 2$ let $m _{\R^d}$ be the Lebesgue  measure on $\R^d$. Then 
$$m_{\R^d}(A_N) \ge (\frac{15^d}{16^d}-\frac{1}{4^d}) e^{-N}. $$
\end{lem}
The explicit constant $(\frac{15^d}{16^d}-\frac{1}{4^d})$ has no importance to us. All we need is that $m_{\R^d}(A_N) \gg e^{-N}.$ However, the explicit constant simplifies the later work.  We can think of $A_N$ as a subset of the unstable subgroup $U^+$ in $G$ w.r.t. $a$. Although $A_N$ has small volume in $\R^d$, it gets expanded by $\T^N$ to a set of volume $\gg e^{dN}$ which will give us an $(N,s)$-separated set of cardinality $\gg e^{dN}$.
\begin{proof}
We will prove that $m_{\R^d}(A_N') \ge (\frac{15^d}{16^d}-\frac{1}{4^d}) e^{-N}$ where 
\begin{equation}
\label{eqn:A_N'}
A_N'=A_N \cap [\frac{1}{16}e^{-N/d}, e^{-N/d}]^d .
\end{equation}
Assume that $\h(\T^{N}(x_t))>16M.$ So, for some nonzero $(p_1,p_2,...,p_d, q) \in \Z^{d+1}$ with $\gcd (p_1,p_2,...,p_d,q)=1$ and $q>0$ we must have 
\begin{align*} &\| ( p_1, p_2,...,p_d,q) g_t a^{N}\|\\&= \| (p_1M^{1/d}+q \frac{t_1}{M})e^{N/d},
 (p_2M^{1/d}+q \frac{t_2}{M})e^{N/d},..., (p_dM^{1/d}+q \frac{t_d}{M})e^{N/d},
q\frac{1}{M}e^{-N} ) \|\\&<\frac{1}{16M}.
\end{align*}
So, letting $\epsilon=\frac{e^{-N/d}}{16M^{(d+1)/d}}$ we have
\begin{equation}
\label{eqn:badpoints}
|p_i+q \frac{t_i}{M^{(d+1)/d}}|<\epsilon \text{ for all } i=1,2,...,d \text{ and } q<\frac{e^N}{16}.
\end{equation}

We have $t_i \in [\frac{1}{16} e^{-N/d}, e^{-N/d}]$. For a fixed $q$, we will calculate the Lebesgue measure of $(t_1,t_2,...,t_d)\in [\frac{1}{16} e^{-N/d}, e^{-N/d}]^d$ for which \eqref{eqn:badpoints} hold for some $p_i$'s.

We have
$$q \frac{t_i}{M^{(d+1)/d}} \in [q \epsilon, 16 q \epsilon].$$
If $16 q \epsilon \le \frac{1}{2}$ then $(p_1,p_2,...,p_d)=0$ and since we only need to consider the primitive vectors in $x_t$ we have $q=1$. In this case, $q \frac{t_i}{M^{(d+1)/d}} \in [ \epsilon, 16 \epsilon]$ and hence \eqref{eqn:badpoints} does not hold. So, we can assume that 
$$16 q \epsilon > \frac{1}{2}.$$
We note that $q \frac{t_i}{M^{(d+1)/d}} $ must be in the $\epsilon$-neighborhood of an integer point. If $16 q \epsilon \in (1/2,1) $ then $[q \epsilon, 16 q \epsilon] $ does not contain any integers and only possible way for \eqref{eqn:badpoints} to hold is when $q \frac{t_i}{M^{(d+1)/d}} $ is in $(1-\epsilon,1+\epsilon)$ so that $t_i $ must be in 
$$(\frac{(1-\epsilon)M^{(d+1)/d}}{q},\frac{(1+\epsilon)M^{(d+1)/d}}{q}).$$
Thus, for a fixed $q \in (\frac{1}{32 \epsilon},\frac{1}{16\epsilon})$ we have that the Lebesgue measure of points that satisfy \eqref{eqn:badpoints} is 
$$\le \left (\frac{2\epsilon M^{(d+1)/d}}{q}\right )^d=\frac{2^d \epsilon^d M^{d+1}}{q^d}.$$
Now, for $16q \epsilon \ge1$ we have that $[q\epsilon, 16q\epsilon]$ has at most $\le 15q \epsilon+1$ integer points. Thus, there could be $\le 15q\epsilon +2$ integers for which $q \frac{t_i}{M^{(d+1)/d}}$ can be $\epsilon$-close for some $t_i$. Since $16q \epsilon \ge 1$ we have $15q\epsilon+2 \le 48q \epsilon.$ Hence, arguing as in the previous case, for a fixed $q \ge \frac{1}{16\epsilon}$ we have that the Lebesgue measure of points satisfying \eqref{eqn:badpoints} is
$$\le \left((48q\epsilon)(2\epsilon)(\frac{M^{(d+1)/d}}{q})\right)^d=96^d \epsilon^{2d}M^{d+1}.$$
Thus, we obtain that the Lebesgue measure of points for which \eqref{eqn:badpoints} hold is
$$\le \sum_{q=\lceil\frac{1}{32\epsilon} \rceil}^{{\lfloor \frac{1}{16\epsilon} \rfloor}} \frac{2^d \epsilon^d M^{d+1}}{q^d}+\sum_{q=\lceil \frac{1}{16\epsilon }\rceil}^{\lfloor \frac{e^N}{16} \rfloor}96^d \epsilon^{2d}M^{d+1}.$$
Since $\epsilon^d=\frac{e^{-N}}{16^dM^{d+1}}$, the above inequality simplifies to
\begin{equation}
\label{eqn:badpointestimate}
\le e^{-N}\left(\sum_{q=\lceil\frac{1}{32\epsilon} \rceil}^{{\lfloor \frac{1}{16\epsilon} \rfloor}} \frac{2^d }{16^d q^d}+\sum_{q=\lceil \frac{1}{16\epsilon }\rceil}^{\lfloor \frac{e^N}{16}\rfloor} \frac{96^d e^{-N}}{16^{2d}M^{d+1}}\right).
\end{equation}
We want to show that, independent of $N$, the term inside the parenthesis is strictly less than 1.
$$\sum_{q=\lceil\frac{1}{32\epsilon} \rceil}^{{\lfloor \frac{1}{16\epsilon} \rfloor}} \frac{2^d }{16^d q^d}\le \sum_{q=\lceil\frac{1}{32\epsilon} \rceil}^{{\lfloor \frac{1}{16\epsilon} \rfloor}} \frac{2^d }{16^d q}\le \frac{1}{8^d \frac{1}{32\epsilon}} (\lfloor \frac{1}{16\epsilon} \rfloor-\lceil\frac{1}{32\epsilon} \rceil)\le \frac{1}{8^d}.$$
On the other hand,
$$\sum_{q=\lceil \frac{1}{16\epsilon }\rceil}^{\lfloor \frac{e^N}{16}\rfloor} \frac{96^d e^{-N}}{16^{2d}M^{d+1}}\le \frac{96^d e^{-N}}{16^{2d}M^{d+1}}\frac{e^N}{16}<\frac{1}{2^{d+4}M^{d+1}}.$$
Together, we see that the inequality \eqref{eqn:badpointestimate} is
$$<(\frac{1}{8^d}+\frac{1}{2^{d+4}M^{d+1}})e^{-N}\le \frac{e^{-N}}{4^d}.$$
Thus, we conclude that
$m_{\R^d}(A_N) \ge m_{\R^d}(A_N')> (\frac{15^d}{16^d}-\frac{1}{4})e^{-N}.$
\end{proof}
From the set $A_N$, in fact from $A_N'$ as in \eqref{eqn:A_N'}, we want to pick about $e^{dN}$ many elements which are not too close to each other so that within $N$ iterations under $\T$ they get apart from each other. For this purpose, let us partition $[\frac{1}{16}e^{-N/d}, e^{-N/d}]^d$ into $\lfloor e^{N} \rfloor^d$ small $d$-cubes of side length $\frac{15}{16}e^{-N(d+1)/d}$. 

Now, consider even smaller $d$-cubes of side length $\frac{13}{16} e^{-N(d+1)/d}$ each lying at the center of one of the small $d$-cubes. We need to find a lower bound for the number of these smaller $d$-cubes that intersect with the set $A_N'$. Each of these $d$-cubes has volume equal to $(\frac{13}{16})^d e^{-N(d+1)}$. Thus, there could be at most
$$\left \lceil \frac{(\frac{1}{4^d}) e^{-N}}{(\frac{13}{16})^d e^{-N(d+1)}}\right \rceil=\left \lceil \frac{4^d}{13^d}e^{dN}\right \rceil$$
many that do not intersect with $A_N'$. Therefore, for $N$ large, at least 
$$\lfloor e^N \rfloor^d-\left \lceil \frac{4^d}{13^d}e^{dN}\right \rceil\ge \frac{1}{13} e^{dN}$$
of these smaller $d$-cubes do intersect with $A_N'$. 

Let us pick one element $t$ from each of these smaller $d$-cubes that is also contained in $A_N'$ and consider the set $S_N'(1)$ of these lattices $x_t=\Gamma g_t$ where $g_t$ is as in \eqref{eqn:g_t}. To simplify notation we let 
\begin{equation}
\label{eqn:S_N'(1)}
S_N'(1)=\{x_1,x_2,...,x_K\}=\{\Gamma g_1,\Gamma g_2,...,\Gamma g_K\}
\end{equation}
where 
$$K=\lfloor \frac{1}{13} e^{dN} \rfloor.$$
We note that for elements $t,t'$ that are picked from different $d$-cubes one has
\begin{equation}
\label{eqn:sep}
\frac{1}{4}e^{-N(d+1)/d} \le \|t-t'\| < \frac{15}{16}e^{-N/d}.
\end{equation}
\begin{prop}
\label{prop:sep} For a given large $N$ the set $S_N'(1)=\{x_1,x_2,...,x_K\}$ has the following properties:
 \begin{enumerate}
 \item $\h(\T^l(x_i))\ge M$ for $l \in [1,N]$ and $i \in [1,K],$
 \item $\h(x_i)<M$ and $\h(\T^{N}(x_i))< 16M$ for any $i \in [1,K]$,
 \item for $i\neq j$ we have $d(g_i,g_j)<\frac{30}{16}e^{-N/d}$ and $d(\T^N(g_i),\T^N(g_j))\ge \frac{1}{8M}.$
\end{enumerate}
\end{prop}
\begin{proof}
Let $x_i=x_t=\Gamma g_t$ for some $t=(t_1,t_2,...,t_d) \in [\frac{1}{16} e^{-N/d}, e^{-N/d}]^d$ (cf. \eqref{eqn:g_t}). It is easy to see that  $x_t \in X_{< M}$. On the other hand, by construction $t \in A_N$ so that $\T^{N}(x_t) \in X_{<16M}$. 

Now, consider the vector $v=(\frac{t_1}{M} ,\frac{t_2}{M},...,\frac{t_d}{M},\frac{1}{M} ) \in x_t.$ We have $$\T(v)=(\frac{t_1e^{1/d}}{M} ,\frac{t_2e^{1/d}}{M},...,\frac{t_de^{1/d}}{M},\frac{e^{-1}}{M} )$$
so that
$$\|\T(v)\|\le\max \{\frac{ e^{-(N-1)/d}}{M},\frac{e^{-1}}{M}\}< \frac{1}{M}.$$
Also,
$$\T^N(v)=(\frac{t_1e^{N/d}}{M} ,\frac{t_2e^{N/d}}{M},...,\frac{t_de^{N/d}}{M},\frac{e^{-N}}{M} )$$
which implies
$$\|\T^N(v)\| \le \max\{\frac{1}{M},\frac{e^{-N}}{M}\}\le \frac{1}{M}.$$
Since the function $\|\T^l(v)\|$ in $l$ has only one critical point we conclude that for $l=1,2,...,N$
$$\h(\T^l(x_t)) \ge M.$$
Let $x_j$ be another element and let $t'\in [\frac{1}{16} e^{-N/d}, e^{-N/d}]^d$ be such that $x_j=x_{t'}=\Gamma g_{t'}.$ 
From \eqref{eqn:sep} together with left invariance of the metric we have
$$d(\T^N(g_t),\T^N(g_{t'}))=d(a^N a^{-N}g_t a^N,a^N a^{-N}g_t' a^N)\ge \frac{\|t-t'\|}{2M}e^{N(d+1)/d}\ge\frac{1}{8M}.$$ 
The fact that $d(g_i,g_j)<\frac{30}{16}e^{-N/d}$ follows from \eqref{eqn:sep} also.
\end{proof}
Our main tool for the construction of lattices is the shadowing lemma:
\begin{lem}[Shadowing lemma]
\label{lem:shadow}
Let $\epsilon \in (0, \eta_0/(3c_0))$ be given. If $d(x_{-},x_{+})< \epsilon$ for some $x_-,x_+ \in X$ then there exists $y\in X$ such that
\begin{enumerate}
\item $d(\T^l(y),\T^l(x_{-}))< 2c_0\epsilon e^{l(d+1)/d}$ for all $l\leq 0$ and
\item $d(\T^l(y),\T^l(x_{+}))< 3c_0\epsilon $ for all $ l \geq 0$.
\end{enumerate}
Moreover, there exists $c$ in the centralizer $C$ of $a$ with $d(c,1)<3c_0\epsilon$ such that
$d(\T^l(y),\T^l(x_{+}c))< 6c_0^2\epsilon e^{-l(d+1)/d}$ for all $ l \geq 0$.
\end{lem}
\begin{proof}
We have $x_{-}=x_{+}g$ for some $g=(g_{ij}) \in \SL(d+1,\R)$ with $d(g,1)<\epsilon$. Consider 
$$u^+= \left( \begin{array}{ccccc}  1&0 &...&0&0 \\
 0&1 &...&0&0\\
\vdots&\vdots&&\vdots&\vdots\\
0&0 &...&1&0\\
u_1 &u_2 &...&u_d&1  \end{array} \right)$$
 and let $y=x_{-}u^+$. For $\|(u_1,u_2,\dots,u_d)\|<2c_0\epsilon$ we have
\begin{align*}
d(\T^l(y),\T^l(x_{-}))&=d(x_{-} u^+a^l,x_{-}a^l)\\&=d(x_{-}a^l a^{-l}u^+a^l,x_{-}a^l)\\
& \leq d\left(\left( \begin{array}{cccccc}  1&0 &...&&0&0 \\
 0&1 &...&&0&0\\
\vdots&\vdots&&&\vdots&\vdots\\
0&0 &...&&1&0\\
u_1e^{l(d+1)/d} &u_2e^{l(d+1)/d} &...&&u_de^{l(d+1)/d}&1  \end{array} \right),1\right)\\
&<\|(u_1,u_2,\dots,u_d)\|e^{l(d+1)/d}<2c_0 \epsilon e^{l(d+1)/d}.
\end{align*}
This establishes part $(i)$. Now, we let

 $g':=gu^+$
 $$=\left( \begin{array}{cccccc}   g_{11}+g_{1(d+1)}u_1&...&&g_{1d}+g_{1(d+1)}u_d&g_{1(d+1)} \\
 g_{21}+g_{2(d+1)}u_1&...&&g_{2d}+g_{2(d+1)}u_d&g_{2(d+1)}\\
\vdots&&&\vdots&\vdots\\
.&...&&.&.\\
 g_{(d+1)1}+g_{(d+1)(d+1)}u_1&...&&g_{(d+1)d}+g_{(d+1)(d+1)}u_d&g_{(d+1)(d+1)}  \end{array} \right).$$
 Since $d(g,1)<\epsilon$, from \eqref{eqn:metric} we have that 
 $$|g_{(d+1)(d+1)}-1| \le \|g-1\|<c_0d(g,1)<1/2.$$
  In particular, $g_{(d+1)(d+1)} \neq 0$. Letting $u_i=-\frac{g_{(d+1)i}}{g_{(d+1)(d+1)}}$ for $i=1,2,...,d$ we can make sure that the unstable part with respect to $a$ is $0$. For any $i \in [1,d]$ we have $|g_{(d+1)i}| \le  \|g-1\| < c_0 \epsilon$. Hence, we have 
 $$||(u_1,u_2,\dots,u_d)||=\frac{1}{|g_{(d+1)(d+1)}|}\max_i \{|g_{(d+1)i}|\}< \frac{c_0\epsilon}{1/2}=2c_0\epsilon.$$
  Now,
  $$d(\T^l(y),\T^l(x_{+}))=d(\T^l(x_{+}gu^+),\T^l(x_{+}))= d(x_{+}a^la^{-l}g'a^l,x_{+}a^l)\leq d(a^{-l}g'a^l,1).$$
  Since unstable part of $g'$ is $0$, for $l\geq 0$ we obtain
  $$d(\T^l(y),\T^l(x_{+}))\leq d(g',1)=d(g u^+,1) \le d(u^+,1)+d(1,g)<  \|u^+\|+ \epsilon < 3c_0\epsilon.$$
  For the last part, let $$c=\left( \begin{array}{cccccc}   g_{11}+g_{1(d+1)}u_1&...&&g_{1d}+g_{1(d+1)}u_d&0 \\
 g_{21}+g_{2(d+1)}u_1&...&&g_{2d}+g_{2(d+1)}u_d&0\\
\vdots&&&\vdots&\vdots\\
.&...&&.&0\\
 0&...&&0&g_{(d+1)(d+1)}  \end{array} \right),$$
 then we have that $c \in C$ with $d(c,1)\le d(g',1)<3c_0\epsilon$, and hence $d(c^{-1},1)<3c_0\epsilon.$ On the other hand, if we let $u^-=c^{-1}g'$ then, $u^- \in U^-$ and
 $$\|u^--1\| <c_0 d(u^-,1) \le  c_0 d(g',1)+c_0d(1,c) < 6c_0^2 \epsilon.$$
 Thus, $d(\T^l(y),\T^l(x_{+}c))=d(x_+gu^+a^l,x_{+}ca^l)=d(x_+g'a^l,x_{+}ca^l)\leq d(g'a^l,ca^l)= d(a^{-l}c^{-1}g'a^l,1)=d(a^{-l}u^-a^l,1)<\|u^--1\| e^{-l(d+1)/d} < 6c_0^2\epsilon e^{-l(d+1)/d} $.
 \end{proof}
 \section{Construction}
 \label{sec:construction}
In this section we construct the set $S_N$ mentioned in the introduction with the properties as in Theorem \ref{thm:main}. Repeatedly using both the shadowing lemma and $K$ lattices constructed in the previous section we obtain more and more lattices that in the limit gives the set $S_N$. 

Recall the set $S_N'(1)$ constructed in \S~\ref{sec:initialstep} (see \eqref{eqn:S_N'(1)}). Let $M'>0$ be a height that depends on $N$ such that for any $x_i \in S_N'(1)$ and for any $l=0,1,...,N$ we have $\T^l(x_i) \in X_{<M'}$. Recall that $\delta>0$ is an injectivity radius for $X_{<17M}$ with $\delta<\min\{\frac{1}{8M},\eta_0\}.$ Now, let $\eta\in (0,\delta)$ be such that $2\eta$ is an injectivity radius of $X_{<M'}$. Recall that $K=\lfloor \frac{1}{13} e^{dN} \rfloor$. We will prove Theorem~\ref{thm:main} with choice of $s=\eta/e^2$ and with the choice of $M'$ as defined above.
 
 Theorem~\ref{thm:main} follows from the following proposition. 
\begin{prop}
\label{prop:main}
As before, let $N$ be sufficiently large. For any positive integer $m$, there is a subset $$S_N'(m)=\left\{x_{i_1i_2...i_m}:i_1,i_2,...,i_m\in\{1,2,...,K\}\right\}$$ of $X_{<M}$ with the following properties:
\begin{enumerate}
\item for any $x\in S_N'(m)$ we have 
$$|\{l\in [0,m N+ (m-1)N']:\T^l(x)\in X_{\ge M/(c_0+1)}\}|\ge m N, $$
\item for any $x \in S_N'(m)$ we have $\T^{mN+(m-1)N'}(x) \in X_{<17M}$,
\item for any distinct $x_{i_1i_2...i_m},x_{j_1j_2...j_m}\in S_N'(m)$, say $i_n \neq j_n$, there exist $g,h \in G$ such that 
$$\T^{(n-1)(N+N')}(x_{i_1i_2...i_m})=\Gamma g\text{ and } \T^{(n-1)(N+N')}(x_{j_1j_2...j_m})=\Gamma h$$
 with $d(\Gamma g, \Gamma h)=d(g,h)$  and that
    \begin{align*}
    d(\T^N(g),\T^N(h))&>\delta -\frac{\delta}{3^4} \text{ if }  n=m \text{ and }\\
    d(\T^N(g),\T^N(h))&>\delta -\delta \sum_{l=3}^{m-n+2} 3^{-l} \text{ if } n\in[1,m).
    \end{align*}
\end{enumerate}
Moreover, we can make sure that for $x_{i_1i_2...i_m}\in S_N'(m)$ and for $x_{i_1i_2...i_{m+1}}\in S_N'(m+1)$ we have $d(x_{i_1i_2...i_m},x_{i_1i_2...i_{m+1}})<\delta e^{-m}.$
\end{prop}
To derive Theorem~\ref{thm:main} from Proposition~\ref{prop:main} we need the lemma below which helps us to determine when two lattices get separated.
\begin{lem}
 \label{lem:group}
 For $\Gamma g, \Gamma h \in X$ with $\T^l(\Gamma g),\T^l(\Gamma h) \in X_{<M'}$ in $[0,N]$ assume that $d(g,h)<\frac{\eta}{e^{2}}$ and $d(\T^{N}(g),\T^{N}(h))\ge\frac{\eta}{e^{2}}$. Then $\Gamma g, \Gamma h$ is $(N,\frac{\eta}{e^{2}})$-separated, that is, there exists $l\in [1, N]$ with $d(\T^l(\Gamma g),\T^l(\Gamma h))\ge \frac{\eta}{e^{2}}.$
 \end{lem}
 \begin{proof}
 Since we have $d(g,h)<\frac{\eta}{e^{2}}$ and that $d(\T^{N}(g),\T^{N}(h))>\frac{\eta}{e^{2}}$, there exists $l \in [1,N]$ such that 
 $$d(\T^{l-1}(g),\T^{l-1}(h)) < \frac{\eta}{e^{2}} \le d(T^l(g),T^l(h)).$$
We have $d(\T(g),\T(h))=d(a^{-1}h^{-1}g a,1)=d(a^{-1}u^+a a^{-1}u^-c a,1)$. On the other hand, we note that any two elements of the unstable subgroup with respect to $a$ gets expanded at most by the factor of $e^{(d+1)/d}$ under the action of $\T$. Together with triangle inequality we have 
\begin{align*}
d(a^{-1}u^+a a^{-1}u^-c a,1)&\le d(a^{-1}u^+a a^{-1}u^-c a,a^{-1}u^+a)+d(a^{-1}u^+a,1)\\
&=d(a^{-1}u^-c a,1)+d(a^{-1}u^+a,1)\\
&\le d(u^-c,1)+e^{(d+1)/d}d(u^+,1)\\
&\le e^{2}(d(u^-c,1)+d(u^+,1))\\
&\le 2e^{2}d(u^+u^-c,1).
\end{align*}
Thus, $d(\T^l(g),\T^l(h)) \le 2e^{2}d(\T^{l-1}(g),\T^{l-1}(h))<2\eta$. On the other hand, $\T^l(\Gamma g)$, $\T^l(\Gamma h)$ are in $X_{<{M'}}$ and $2\eta$ is an injectivity radius of $X_{<{M'}}$. Hence, $$d(\T^l(\Gamma g),\T^l(\Gamma h))=d(\T^l(g),T^l(h)) \ge \frac{\eta}{e^{2}}. $$ 
\end{proof}
\begin{proof}[Proof of Theorem~\ref{thm:main}]
For any $m$ let us pick a set 
$$S_N'(m)=\left\{x_{i_1i_2...i_m}:i_1,i_2,...,i_m\in\{1,2,...,K\}\right\}$$
 as in Proposition~\ref{prop:main}. Also, assume for $x_{i_1i_2...i_m}\in S_N'(m)$ and for $x_{i_1i_2...i_{m+1}}\in S_N'(m+1)$ we have $d(x_{i_1i_2...i_m},x_{i_1i_2...i_{m+1}})<\delta e^{-m}.$ If we fix a sequence $\{i_l\} \subset \{1,2,...,K\}^{\N}$, then the sequence $\{x_{i_1},x_{i_1i_2},x_{i_1i_2i_3},...\}$ becomes a Cauchy sequence and hence converges. So, we let 
 $x_{\{i_l\}}=\lim_{n \to \infty}x_{i_1i_2...i_m}.$ 
 Varying the sequence $\{i_l\}$ we define the set 
 $$S_N=\left\lbrace x_{\{i_l\}}:\{i_l\} \subset \{1,2,...,K\}^{\N}\right\rbrace .$$
 Also, define subsets $S_N(m)$'s of $S_N$ 
 $$S_N(m)=\left\lbrace x_{\{i_l\}}:\{i_l\} \subset \{1,2,...,K\}^{\N} \text{ with } i_l=1 \text{ for all } l>m \right\rbrace .$$
 By definition of $S_N(m)$ and by $(i)$ of Proposition~\ref{prop:main}, for any $x_{\{i_l\}}\in S_N(m)$ we have 
 $$|\{l\in [0,m N+ (m-1)N']:\T^l(\{x_I\})\in X_{\ge M/(c_0+1)}\}|\ge m N.$$
As for part $(ii)$, again from the construction of the set $S_N(m)$ and from $(iii)$ of Proposition \ref{prop:main} we conclude that for any distinct $x_{\{i_l\}},x_{\{j_l\}} \in S_N(m)$, say $i_n \neq j_n$, there exist $g,h \in G$ with $\T^{(n-1)(N+N')}(x_{\{i_l\}})=\Gamma g,\T^{(n-1)(N+N')}(x_{\{j_l\}})=\Gamma h$ and $d(\Gamma g, \Gamma h)=d(g,h)$ such that
 $$d(\T^N(g),\T^N(h))>\delta -\delta \sum_{l=3}^{\infty} 3^{-l}=\frac{17}{18}\delta.$$
If $d(\Gamma g, \Gamma h) \ge \frac{\eta}{e^2}$ then there is nothing to show, if not then from Lemma~\ref{lem:group} for some $s\in [1,N]$ we conclude that $d(\T^s(\Gamma g),\T^s(\Gamma h))\ge \frac{\eta}{e^2}$ since $\frac{\eta}{e^2} < \frac{17}{18}\delta$. Thus, for some $s\in [1,N]$ we have
$$d\left( \T^{(n-1)(N+N')+s}(x_{\{i_l\}}),\T^{(n-1)(N+N')+s}(x_{\{j_l\}})\right) \ge \frac{\eta}{e^2}$$  
and hence the set $S_N(m)$ is $(m N+ (m-1)N',\eta/e^{2})$-separated since $n \le m$. This concludes the proof.
\end{proof}
Now, we will make use of what we obtained in the previous section to prove Proposition~\ref{prop:main}.
\begin{proof}[Proof of Proposition~\ref{prop:main}]
We inductively prove $(ii)$ and $(iii)$ and briefly discuss how these arguments imply $(i)$. Let us fix some large $N$.

For $m=1$ let $S_N'(1)=\{x_1,x_2,...,x_{K}\}$ be the set as in Proposition \ref{prop:sep}. It is clear that $(i)$ and $(ii)$ are satisfied. Let $x_i=\Gamma g_i, x_j=\Gamma g_j$ be distinct elements (cf. \eqref{eqn:S_N'(1)}). Then letting $g=g_i$ and $h=g_j$ we obtain $(iii)$ since the part $(iii)$ of Proposition~\ref{prop:sep} gives
$$d(\T^N(g_i),\T^N(g_j))\ge \frac{1}{8M}>\delta.$$
Now, assume that the proposition holds for $m=k \ge 1$, we have the set $S_N'(k)=\{x_{i_1i_2...i_k}:i_1,i_2,...,i_k=1,...,K\}$. Let us construct the set $S_N'(k+1)$.

For any $x_{i_1i_2...i_k} \in S_N'(k)$, we have $\T^{kN+(k-1)N'}(x_{i_1i_2...i_k})\in X_{<17M}$. Hence, applying Lemma~\ref{lem:join} we have that for $x_j$ there exists $z$ with
$$d(\T^{kN+(k-1)N'}(x_{i_1i_2...i_k}),z)<\delta/(c_0^3 3^9) \text{ and } d(x_j,\T^{N'}(z))<\delta/(c_0^3 3^9).$$
Now, we apply shadowing lemma with $x_{-}=\T^{kN+(k-1)N'}(x_{i_1i_2...i_k})$ and $x_{+}=z$ and $\epsilon=\delta/(c_0^3 3^9)$. There exists $y$ such that
\begin{align}
\label{eqn:yNegl}
d(\T^l(y),\T^l(T^{kN+(k-1)N'}(x_{i_1i_2...i_k})))&< \frac{\delta}{c_0^2 3^8} e^{l(d+1)/d} \text{ for } l\leq 0 \text{ and}\\
d(\T^l(y),\T^l(z))&< \frac{\delta}{c_0^2 3^8}  \text{ for }  l \geq 0.
\end{align}
We have $d(x_j,\T^{N'}(y))<d(x_j,\T^{N'}(z))+d(\T^{N'}(z),\T^{N'}(y))<\delta/(c_0^4 3^9)+\delta/(c_0^2 3^8)<\delta/(c_0^2 3^7).$ We apply shadowing lemma once more with $x_{-}=\T^{N'}(y)$ and $x_{+}=x_j$ and $\epsilon=\delta/(c_0^2 3^7)$. There exists $y'$ such that
\begin{align}
\label{eqn:y'Negl}
d(\T^l(y'),\T^l(\T^{N'}(y)))&< \frac{\delta}{c_0 3^6} e^{l(d+1)/d} \text{ for } l\leq 0 \text{ and }\\
\label{eqn:Posl}
d(\T^l(y'),\T^l(x_j))&< \frac{\delta}{c_0 3^6} \text{ for } l \ge 0
\end{align}
Also, there exists $c_j\in C$ with $d(c_j,1)<\frac{\delta}{c_03^6}$ such that
\begin{equation}
\label{eqn:constantc_j}
d(\T^l(y'),\T^l(x_jc_j))< \frac{\delta}{3^5} e^{-l(d+1)/d}  \text{ for }  l \geq 0 
\end{equation}
Now we let $x_{i_1i_2...i_kj}=\T^{-k(N+N')}(y')$ and varying $j$ we obtain the set 
$$S_N'(k+1)=\{x_{i_1i_2...i_kj} : j \in \{1,2,...,K\}\}.$$
Let us justify part $(ii)$ first.  Let us fix some $j=1,2,...,K.$ Recalling that $x_{i_1i_2...i_kj}=\T^{-k(N+N')}(y')$ we obtain from \eqref{eqn:Posl} with $l=N$ that
$$d(\T^{(k+1)N+kN'}(x_{i_1i_2...i_k j}),\T^N(x_j))< \frac{\delta}{c_0 3^6}.$$
Moreover, from Proposition~\ref{prop:sep} we have $\T^N(x_j) \in X_{<16M}$ so that
$$\h(\T^{(k+1)N+kN'}(x_{i_1i_2...i_k j}) )\le \frac{\h(\T^N(x_j))}{1-\frac{\delta}{ 3^6}}<17M.$$

To prove $(iii)$ let us consider any distinct pairs $x_{i_1i_2...i_ki_{k+1}}$ and $x_{j_1j_2...j_kj_{k+1}}$ in $S_N'(k+1).$ First, assume that $i_{k+1}\neq j_{k+1}$ and let $g,h \in G$ be such that 
$$\T^{k(N+N')}(x_{i_1i_2...i_ki_{k+1}})=\Gamma g,\,\, \T^{k(N+N')}(x_{j_1j_2...j_kj_{k+1}})=\Gamma h$$ 
with
\begin{multline}
\label{eqn:c_{i_{k+1}}}
d(\T^{k(N+N')+N}(x_{i_1i_2...i_ki_{k+1}}c_{i_{k+1}}),\T^N(x_{i_{k+1}}))\\=d(\T^N(gc_{i_{k+1}}),\T^N(g_{i_{k+1}}))
<\frac{\delta}{3^5}e^{-N(d+1)/d} \text{ and }
\end{multline}
\begin{multline}
\label{eqn:c_{j_{k+1}}}
d(\T^{k(N+N')+N}(x_{j_1j_2...j_kj_{k+1}}c_{j_{k+1}}),\T^N(x_{j_{k+1}}))\\=d(\T^N(hc_{j_{k+1}}),\T^N(g_{j_{k+1}}))
<\frac{\delta}{3^5}e^{-N(d+1)/d}
\end{multline}
for some $c_{i_{k+1}},c_{j_{k+1}}\in C$ with $d(c_{i_{k+1}},1)<\frac{\delta}{c_0 3^6}$ and $d(c_{j_{k+1}},1)<\frac{\delta}{c_0 3^6}$ as in \eqref{eqn:constantc_j}. Thus, we have
$$d(g_{i_{k+1}},gc_{i_{k+1}})<\frac{\delta}{3^5} \text{ and } d(g_{j_{k+1}},hc_{j_{k+1}})<\frac{\delta}{3^5}.$$
We also note from Proposition~\ref{prop:sep} that $d(g_{i_{k+1}},g_{j_{k+1}})<\frac{30}{16}e^{-N/d}.$ Thus, for $N$ large enough we get
\begin{align*}
&d(g,h)\\
&<d(g,g c_{i_{k+1}})+d(g c_{i_{k+1}},g_{i_{k+1}})+d(g_{i_{k+1}},g_{j_{k+1}})+d(g_{j_{k+1}},hc_{j_{k+1}})+d(h c_{j_{k+1}},h)\\
&<\frac{\delta}{3^6}+\frac{\delta}{3^5}+\frac{30}{16}e^{-N/d}+\frac{\delta}{3^5}+\frac{\delta}{3^6}\\
&<\frac{\delta}{3^4}.
\end{align*}
In particular, $d(\Gamma g,\Gamma h)=d(g,h)$ since $\delta$ is an injectivity radius for $X_{<17M}$.
On the other hand, from Proposition~\ref{prop:sep} we know that 
$$d(\T^N(g_{i_{k+1}}),\T^N(g_{j_{k+1}}))>\frac{1}{8M}>\delta.$$
 So, together with \eqref{eqn:c_{i_{k+1}}} and \eqref{eqn:c_{j_{k+1}}} we conclude that
\begin{align*}
&d(\T^N(g),\T^N(h))\\
&>d(\T^N(g_{i_{k+1}}),\T^N(g_{j_{k+1}}))-d(\T^N(g_{i_{k+1}}),\T^N(g))-d(\T^N(g_{j_{k+1}}),\T^N(h))\\
&>\delta-\frac{\delta}{3^5}e^{-N(d+1)/d}-\frac{\delta}{c_0 3^6}-\frac{\delta}{3^5}e^{-N(d+1)/d}-\frac{\delta}{c_0 3^6}\\
&>\delta -\frac{\delta}{3^4}.
\end{align*}
Now, assume that $i_n\neq j_n$ for some $n\le k$. By replacing $l$ in \eqref{eqn:yNegl} by $l-(k-n)(N+N')$ we obtain
\begin{multline}
\label{eqn:yNegl2}
d(\T^{l-(k-n)(N+N')}(y),\T^{l+n(N+N')-N'}(x_{i_1i_2...i_k}))\\
< \frac{\delta}{c_0^23^8} e^{(l-(k-n)(N+N'))(d+1)/d}  \text{ for } l \le 0.
\end{multline}
On the other hand, if we replace $l$ in \eqref{eqn:y'Negl} by $l-(k-n)(N+N')-N'$ we get
\begin{multline}
\label{eqn:y'Negl2}
d(\T^{l-(k-n)(N+N')-N'}(y'),\T^{l-(k-n)(N+N')}(y))\\
< \frac{\delta}{c_0 3^6} e^{(l-(k-n)(N+N')-N')(d+1)/d} \text{ for } l\leq 0.
\end{multline}
Thus, \eqref{eqn:yNegl2} and \eqref{eqn:y'Negl2} together with the triangular inequality give
\begin{multline*}
d(\T^{l-(k-n)(N+N')-N'}(y'),\T^{l+n(N+N')-N'}(x_{i_1i_2...i_k}))\\
< \frac{\delta}{c_0 3^5} e^{(l-(k-n)(N+N')-N')(d+1)/d} 
\end{multline*}
for $l\le 0$ where $y'=\T^{-k(N+N')}(x_{i_1i_2...i_k j})$ for $j=1,2,...,K$. Thus, we have
\begin{multline}
\label{eqn:ii}
d(\T^{n(N+N')-N'+l}(x_{i_1i_2...i_k}),\T^{n(N+N')-N'+l}(x_{i_1i_2...i_{k+1}}))\\
<\frac{\delta}{c_0 3^5}e^{(l-(k-n)(N+N'))(d+1)/d}
\end{multline}
and
\begin{multline}
\label{eqn:jj}
d(\T^{n(N+N')-N'+l}(x_{j_1j_2...j_k}),\T^{n(N+N')-N'+l}(x_{j_1i_2...j_{k+1}}))\\
<\frac{\delta}{c_0 3^5}e^{(l-(k-n)(N+N'))(d+1)/d}.
\end{multline}
  Now, from the induction hypothesis we have that there are $g',h'$ with
$$\T^{n(N+N')}(x_{i_1i_2...i_k})=\Gamma g',\,\,\, \T^{n(N+N')}(x_{j_1j_2...j_k})=\Gamma h'$$
 such that $d(\Gamma g',\Gamma h')=d(g',h')$ and that
\begin{align*}
    d(\T^N(g'),\T^N(h'))&>\delta -\frac{\delta}{3^4} \text{ if }  n=k \text{ and }\\
    d(\T^N(g'),\T^N(h'))&>\delta -\delta \sum_{l=3}^{k-n+2} 3^{-l} \text{ if } n\in[1,k).
    \end{align*}
 Let $g,h \in G$ be such that
  $$\T^{(n-1)(N+N')}(x_{i_1i_2...i_{k+1}})=\Gamma g \text{ and } \T^{(n-1)(N+N')}(x_{j_1j_2...j_{k+1}})=\Gamma h$$
   with
  \begin{align*}
  d(g,g')&<\frac{\delta}{c_0 3^5}e^{[-(k-n)(N+N')-N](d+1)/d},\\
  d(h,h')&<\frac{\delta}{c_0 3^5}e^{[-(k-n)(N+N')-N](d+1)/d}.
  \end{align*}
   This can be done using \eqref{eqn:ii} and \eqref{eqn:jj} with $l=-N$. In particular,
   \begin{align*}
  d(\T^N(g),\T^N(g'))&<\frac{\delta}{c_0 3^5}e^{-(k-n)(N+N')(d+1)/d},\\
  d(\T^N(h),\T^N(h'))&<\frac{\delta}{c_0 3^5}e^{-(k-n)(N+N')(d+1)/d}.
  \end{align*}
  Also, since by construction 
  $$\T^{(n-1)(N+N')}(x_{i_1i_2...i_{k+1}}), \T^{(n-1)(N+N')}(x_{j_1j_2...j_{k+1}})\in X_{<17M}$$
  and since $\frac{\delta}{3^5}e^{[-(k-n)(N+N')-N](d+1)/d}$ is less than the injectivitiy radius $\delta$ for $X_{<17M}$ we have
     \begin{align*}
   d\left(\T^{(n-1)(N+N')}(x_{i_1i_2...i_{k+1}}),\T^{(n-1)(N+N')}(x_{i_1i_2...i_k})\right)&=d(g,g')\text{ and }\\
   d\left(\T^{(n-1)(N+N')}(x_{j_1j_2...j_{k+1}}),\T^{(n-1)(N+N')}(x_{j_1j_2...j_k})\right)&=d(h,h').
  \end{align*}
  Now, if $n=k$ then 
  \begin{align*}
  d(\T^N(g),\T^N(h))&\geq d(\T^N(g'),\T^N(h'))-d(\T^N(g'),\T^N(g))-d(\T^N(h'),\T^N(h))\\
  &> \delta-\frac{\delta}{3^4}-\frac{\delta}{c_0 3^5}-\frac{\delta}{c_0 3^5}\\
  &=\delta-\frac{\delta}{3^3}\\
  &=\delta-\delta \sum_{l=3}^{k+1-n+2} 3^{-l}.
  \end{align*}
  Otherwise, if $n<k$ then
\begin{align*}
 d(\T^N(g),\T^N(h)) &\geq d(\T^N(g'),\T^N(h'))-d(\T^N(g'),\T^N(g))-d(\T^N(h'),\T^N(h))\\
 &>\delta -\delta \sum_{l=3}^{k-n+2} 3^{-l}-2\frac{\delta}{c_0 3^5}e^{-(k-n)(N+N')(d+1)/d} \\
 &>\delta -\delta \sum_{l=3}^{k-n+2} 3^{-l}-\delta\cdot 3^{-(k-n+3)}\\
 &=\delta -\delta \sum_{l=3}^{k+1-n+2} 3^{-l}.
 \end{align*}
 This concludes the proof of $(iii)$ for $n=k+1$ and the inductive argument. 
 
 Now, we will briefly point out why $(i)$ holds. Clearly it is true for the elements of $S_N'(1)$ as suggested by Proposition~\ref{prop:sep}. In the inductive step, to estimate the distance between the elements of $S_N'(m)$ and $S_N'(m+1)$ under the action of $\T$ we made use of \eqref{eqn:c_{i_{k+1}}}, \eqref{eqn:c_{j_{k+1}}} , \eqref{eqn:ii}, and \eqref{eqn:jj} and obtained part $(iii)$. Arguing in the same way, we can inductively prove for any $m \ge 1$ and for any $x \in S_N'(m)$ that
 $$d(\T^{l+n(N+N')}(x), \T^l(x_j))<\delta \sum_{k=3}^{m-n+3} 3^{-k}$$
 for some $x_j \in S_N'(1)$ and for $l \in [n(N+N'),(n+1)N+nN')]$ with $n \le m$. In particular,
 $$d(\T^{l+n(N+N')}(x), \T^l(x_j))<\delta \sum_{k=3}^{\infty} 3^{-k}=\frac{\delta}{18}$$
 for some $x_j \in S_N'(1)$ and for $l \in [n(N+N'),(n+1)N+nN')]$. Together with $(i)$ of Proposition~\ref{prop:sep} we obtain
 $$\h(\T^{l+n(N+N')}(x)) \ge \frac{\h(\T^l(x_j))}{\frac{c_0 \delta}{18}+1}> \frac{M}{c_0+1}$$
 for $l \in [n(N+N'),(n+1)N+nN')]$. This justifies $(i)$.
 
 Finally, from \eqref{eqn:ii} with $n=1$ and $l=-N$ we have
 \[d(x_{i_1i_2...i_k},x_{i_1i_2...i_{k+1}})<\frac{\delta}{c_0 3^5}e^{(-N-(k-1)(N+N'))(d+1)/d}<\delta e^{-k}\]
which concludes the proof.
\end{proof}
\bibliographystyle{plain}
\bibliography{mybib}

\end{document}